\documentclass[pdflatex,sn-mathphys-num]{sn-jnl}
\usepackage{bm}
\usepackage{amsmath,amssymb,amsthm,mathtools}
\usepackage{enumitem}
\usepackage{hyperref}
\newtheorem{theorem}{Theorem}
\newtheorem{lemma}[theorem]{Lemma}
\newtheorem{proposition}[theorem]{Proposition}
\newtheorem{remark}[theorem]{Remark}

\newtheorem{assumption}[theorem]{Assumption}
\numberwithin{equation}{section}
\numberwithin{theorem}{section}
\numberwithin{table}{section}
\numberwithin{figure}{section}

\begin{document}
	
\title[ReLU complex]{ReLU$^k$ Neural de Rham Complexes}
\author[1]{\fnm{Kaibo} \sur{Hu}}\email{kaibo.hu@maths.ox.ac.uk}

\author*[1]{\fnm{Jindong} \sur{Wang}}\email{jindong.wang@maths.ox.ac.uk}

\author[2]{\fnm{Jinchao} \sur{Xu}}\email{jinchao.xu@kaust.edu.sa}

\affil[1]{\orgdiv{Mathematical Institute}, \orgname{University of Oxford}, \orgaddress{\street{Andrew Wiles Building}, \city{Oxford}, \postcode{OX2 6GG}, \country{United Kingdom}}}

\affil[2]{
  \orgdiv{Applied Mathematics and Computational Science, CEMSE Division},
  \orgname{King Abdullah University of Science and Technology},
  \orgaddress{%
    \city{Thuwal},
    \postcode{23955},
    \country{Saudi Arabia}}}

\abstract{
We construct finite-dimensional de Rham subcomplexes generated by fixed-neuron shallow ReLU$^k$ neural networks, a class of spaces known to provide optimal approximation rates. For neurons of the form $s_i(x)=\omega_i\cdot x+b_i$, we introduce spaces of
neural differential forms: differential $p$-forms whose coefficients are
the ReLU$^k$ ridge functions $\sigma_{k-p}(s_i)$. These spaces are compatible with the exterior derivative because differentiating a ReLU power lowers its order by one, and for each fixed neuron, differentiation amounts to exterior multiplication by the fixed one-form $ d s_i$. Under a linear independence assumption on the lowest-order family $\{\sigma_{k-d}(s_i)\}_{i=1}^n$, the global complex decomposes into independent neuron-wise Koszul complexes. We prove exactness in arbitrary dimension and provide a geometric sufficient condition for the required linear independence. Numerical experiments based on the resulting complex provide evidence of
stable discretizations and of convergence rates consistent with the
underlying approximation theory, and exhibit no spurious modes in
eigenvalue problems considered.
}



\maketitle

\section{Introduction}
\label{sec:introduction}

The de~Rham complex is a fundamental structure in differential geometry and partial differential equations. It organizes scalar and vector fields within the unified language of differential forms, connected by the exterior derivative. This complex not only encodes the identities satisfied by classical differential operators, such as grad, curl, and div, but also reveals the underlying topological structure of the domain through its cohomology. Beyond its geometric interpretation, the de Rham complex plays a central role in characterizing compatibility conditions and topological invariants of differential equations.

Finite element exterior calculus (FEEC) provides a powerful and promising framework for discretizing the de~Rham complex. It has demonstrated that preserving the underlying algebraic and geometric structures is essential for the design and analysis of stable numerical schemes for partial differential equations \cite{arnold2006finite,arnold2010finite,arnold2018finite}.  A central principle in FEEC is that finite-dimensional spaces are organized into subcomplexes of the continuous de~Rham complex, so that discrete differential operators preserve the relations of their continuous counterparts.  Among these structures, exactness and cohomology on a topologically trivial domain are fundamental properties: they rule out spurious closed fields, ensure the correct representation of constraint spaces, and provide a basis for discrete Poincar\'e inequalities and the inf-sup stability of mixed formulations \cite{boffi2013mixed}.  When combined with bounded commuting projections and appropriate discrete norms, it leads to stable numerical schemes with robustness under refinement.

This viewpoint has been significantly extended in recent years.  Nodal and generalized finite element systems provide flexible constructions of de Rham subcomplexes beyond classical polynomial elements \cite{christiansen2018nodal,christiansen2018generalized}.  The Bernstein--Gelfand--Gelfand (BGG) machinery further allows one to generate new complexes from existing ones and has become an important tool in elasticity and related geometric PDEs \cite{cap2001bernstein,arnold2021complexes,cap2024bgg}.  These developments highlight that exact sequences are not merely algebraic artifacts, but structural principles governing stability and compatibility in numerical discretizations.

In parallel, neural network (NN) approximation theory has developed into a mature field for high-dimensional function approximation and numerical analysis of PDEs.  Classical results established universal approximation properties of shallow networks \cite{cybenko1989approximation,hornik1989multilayer}.  More recent advances have quantified approximation rates for deep and shallow ReLU-based architectures and clarified their connection to high-order finite element methods and spectral representations \cite{schwab2019deep,opschoor2020deep,schwab2023deep}.  In particular, there has also been growing interest in the relationship between neural networks and finite element structures, including ReLU representations of finite element
functions and finite-neuron methods \cite{he2020relu,xu2020finite}. For
ReLU$^k$ activations, sharp and high-order approximation results have been
established for suitable smoothness and variation classes
\cite{siegel2022highorder,siegel2024sharp}.

Despite these developments, the algebraic structure of neural network spaces has remained only partially connected to the geometric and topological framework of FEEC.  A related contribution is the work of Longo et al.~\cite{longo2023rham}, which constructs deep neural network emulations of finite element spaces appearing in the discrete de Rham complex, and demonstrates that compatible finite element spaces can be represented by suitable deep network architectures.  However, it does not provide an intrinsic de Rham sequence within a fixed-neuron shallow ReLU space. 

Standard shallow neural network classes are nonlinear approximation spaces,
since both the inner neuron parameters and the outer coefficients are
trainable. By contrast, recent work on linearized shallow ReLU$^k$
approximation considers a fixed-neuron regime, in which the neurons are
prescribed, for instance quasi-uniformly on the unit sphere, and only the
outer coefficients are trained \cite{liu2025integral}. Their results establish Sobolev integral representations and optimal approximation rates for the corresponding linearized spaces, showing that fixed-neuron ReLU$^k$ spaces can retain optimal approximation properties while reducing the approximation problem to a linear one. A related development is the divergence-free linearized ReLU$^k$ construction in \cite{he2026divergence}, which builds exactly divergence-free shallow neural approximation spaces through potential representations. Taken together, these works suggest that linearized shallow ReLU networks can preserve strong approximation power while naturally accommodating structural constraints.

Motivated by this line of work, the goal of the paper is to construct a  de~Rham sequence directly inside a fixed-neuron shallow ReLU$^k$ framework.  Instead of emulating existing finite element spaces by deep networks, we define neural differential form spaces generated by the same affine neurons and prove that they form an exact de Rham complex under a natural lowest-order linear independence assumption. We consider affine ridge functions
\begin{equation}
s_i(x)=\omega_i\cdot x + b_i,
\qquad \omega_i \neq 0,
\end{equation}
and construct finite-dimensional differential form spaces by combining ReLU$^k$ ridge functions with constant-coefficient differential forms.  The key structural observation is that the exterior derivative acts in a decoupled manner on each neuron:
it simultaneously lowers the ReLU degree and applies a fixed algebraic operation determined by the affine function.

The main contribution of this paper is to construct an exact de Rham subcomplex within fixed-neuron shallow ReLU$^k$ spaces. Specifically, we define neural differential form spaces compatible with the exterior derivative. Under a linear independence assumption on the lowest-order ridge functions, we prove that these spaces form the exact sequence
$$
0
\longrightarrow \mathcal{C}^0
\xrightarrow{d}
\mathcal{C}^1
\xrightarrow{d}
\cdots
\xrightarrow{d}
\mathcal{C}^d
\longrightarrow 0.
$$
The proof relies on two elementary mechanisms. First, differentiation transfers linear independence from lower-order ReLU powers to higher-order ones. Second, the exactness with respect to form degree reduces, neuron by neuron, to the Koszul complex generated by the covector $ d s_i$. The linear independence assumption ensures that the contributions from different neurons decouple, so that the neuron-wise exactness yields exactness of the global neural complex. Consequently, the construction provides a discrete kernel–range identity for these neural differential form spaces.

The paper is organized as follows. Section~\ref{sec:prelim} introduces the
neural networks and differential form notation.
Section~\ref{sec:relu-complex} introduces the ReLU$^k$ neural differential-form spaces and the associated de~Rham complex.
Section~\ref{sec:independence} discusses linear independence assumptions and sufficient geometric conditions. Section~\ref{sec:3d} presents the three-dimensional ReLU grad–curl–div complex and proves its exactness as a simple illustration. Section~\ref{sec:main} establishes the neuron-wise Koszul homotopy and proves the exact neural de Rham sequence. Section~\ref{sec:numerical} presents numerical experiments based on the proposed complex. Section~\ref{sec:discussion} concludes the paper with a discussion of the main results and possible directions for future work.
    
\section{Preliminaries}\label{sec:prelim}

We begin by introducing the basic notation and definitions used throughout the paper, including the neural network spaces and differential forms. Let $\Omega\subset\mathbb  R  ^d$ be a bounded contractible domain. For a non-polynomial activation function $\phi:\mathbb  R  \to\mathbb  R  $, the shallow neural network class is
\begin{equation}
\Sigma^{\phi}
:=
\Bigl\{\sum_{i=1}^n a_i\phi(\omega_i\cdot x+b_i): a_i\in\mathbb  R  ,\ \omega_i\in\mathbb  R  ^d,\ b_i\in\mathbb  R ,n\in \mathbb N \Bigr\}.
\end{equation}
This class has the universal approximation property \cite{hornik1989multilayer}. The ReLU$^k$ activation function is given by
\begin{equation}
\sigma_k(t)=\max\{t,0\}^k,
\qquad k\in\mathbb N_0.
\end{equation}
Due to the positive homogeneity $\sigma_k(at)=a^k\sigma_k(t)$ for $a>0$, we adopt the normalized parameterization
\begin{equation}
\theta_i=(\omega_i,b_i)\in \mathbb S ^d\subset\mathbb  R  ^{d+1},
\qquad \widetilde x=(x,1)\in\mathbb  R  ^{d+1},
\end{equation}
so that
\begin{equation}
\sigma_k(\omega_i\cdot x+b_i)=\sigma_k(\theta_i\cdot\widetilde x).
\end{equation}
Given fixed neurons $\Theta=\{\theta_i\}_{i=1}^n\subset\mathbb S ^d$, we define the linearized neural network space
\begin{equation}
L_n^k(\Theta)
:=
\operatorname{span}\{\sigma_k(\theta_i\cdot \widetilde x)\}_{i=1}^n
=\operatorname{span}\{\sigma_k(\omega_i\cdot x+b_i)\}_{i=1}^n.
\end{equation}
When $\Theta$ is understood, we simply write $L_n^k$.

\paragraph{Approximation-theoretic background.}
Recent results in \cite{liu2025integral} show that, for quasi-uniform parameter sets on $\mathbb S ^d$, the spaces $L_n^k$ achieve the optimal approximation rate
\begin{equation}
\inf_{v\in L_n^k}\|u-v\|_{L^2(\Omega)}
\lesssim
n^{-\frac12-\frac{2k+1}{2d}}\,\|u\|_{H^{\frac{d+2k+1}{2}}(\Omega)},
\end{equation}
and that the Sobolev space $H^{\frac{d+2k+1}{2}}(\Omega)$ admits the integral representation
\begin{equation}
H^{\frac{d+2k+1}{2}}(\Omega)
=
\Bigl\{x\mapsto\int_{\mathbb S ^d}\sigma_k(\theta\cdot\widetilde x)\,\psi(\theta)\,d\theta:
\psi\in L^2(\mathbb S ^d)\Bigr\}.
\end{equation}
We use these approximation results as motivation for the construction below.

\paragraph{Constant-coefficient differential forms.}
For $0\le p\le d$, let $\Lambda^p$ denote the space of constant-coefficient $p$-forms on $\mathbb  R  ^d$.
These are the differential forms that can be written as
\begin{equation}
\alpha=\sum_{|I|=p} \alpha_I\,dx_I,
\qquad \alpha_I\in\mathbb  R  ,
\end{equation}
where $I=(i_1,\ldots,i_p)$ ranges over increasing multi-indices and $dx_I=dx_{i_1}\wedge\cdots\wedge dx_{i_p}$. Since the coefficients $\alpha_I$ are constant, one has $d\alpha=0$. In dimension $d=3$, the Euclidean metric identifies
\begin{equation}
\Lambda^0\simeq \mathbb  R  ,
\qquad
\Lambda^1\simeq \mathbb  R  ^3,
\qquad
\Lambda^2\simeq \mathbb  R  ^3,
\qquad
\Lambda^3\simeq \mathbb  R  .
\end{equation}
Fix affine functions
\begin{equation}
s_i(x)=\omega_i\cdot x+b_i,
\qquad i=1,\dots,n,
\end{equation}
with $\omega_i\neq 0$. Then
\begin{equation}
ds_i=\sum_{j=1}^d \omega_{ij}\,dx_j\in\Lambda^1.
\end{equation}
Here $\omega_{ij}$ denotes the $j$-th component of the vector $\omega_i=(\omega_{i1},\ldots,\omega_{id})^T$. Throughout, for an integer $r\ge 0$ we write $\mathcal P _r(\Omega)$ for the restrictions to $\Omega$ of polynomials on $\mathbb  R  ^d$ of degree at most $r$.

\section{The ReLU$^k$ complex}\label{sec:relu-complex}

We now formulate the main complex in arbitrary dimension. Throughout this
section, assume that $k\ge d$. For $0\le p\le d$, define
\begin{equation}
\mathcal C^p
:=
\left\{
\sum_{i=1}^n \sigma_{k-p}(s_i)\,\alpha_i
:\;
\alpha_i\in\Lambda^p
\right\}.
\end{equation}
Thus $\mathcal C^p$ is the fixed-neuron ReLU$^k$ space of
constant-coefficient $p$-forms, with the ReLU degree decreasing by one at
each differential degree.

We first record the basic differential inclusion.

\begin{proposition}[Differential inclusion]\label{prop:inclusion}
For $0\le p\le d-1$, we have
\begin{equation}
d\mathcal C^p\subset \mathcal C^{p+1}.
\end{equation}
More precisely, for every constant-coefficient $p$-form
$\alpha\in\Lambda^p$ and every neuron $s_i$,
\begin{equation}
d\bigl(\sigma_{k-p}(s_i)\alpha\bigr)
=
(k-p)\sigma_{k-p-1}(s_i)\,ds_i\wedge\alpha .
\end{equation}
\end{proposition}

\begin{proof}
Since $\alpha$ has constant coefficients, $d\alpha=0$. Moreover, because
$s_i$ is affine, $ds_i$ is a constant one-form. Hence
\begin{equation}
d\bigl(\sigma_{k-p}(s_i)\alpha\bigr)
=
d\bigl(\sigma_{k-p}(s_i)\bigr)\wedge\alpha
=
(k-p)\sigma_{k-p-1}(s_i)\,ds_i\wedge\alpha .
\end{equation}
The right-hand side belongs to $\mathcal C^{p+1}$, which proves the claim.
\end{proof}

Combining Proposition~\ref{prop:inclusion} with the identity $d\circ d=0$,
we obtain the finite-dimensional differential complex
\begin{equation}\label{eq:relu-complex}
0
\longrightarrow \mathcal C^0
\xrightarrow{d} \mathcal C^1
\xrightarrow{d}
\cdots
\xrightarrow{d} \mathcal C^d
\longrightarrow 0 .
\end{equation}
It remains to prove exactness of this complex.

\section{Linear independence of neurons}\label{sec:independence}
The proof of exactness of the complex requires a linear independence condition on the lowest-order ReLU ridge functions given by the prescribed neurons. In this section, we formulate this requirement and provide a geometric sufficient condition under which it holds. This condition excludes degenerate configurations in which different neurons produce linearly dependent ridge functions, while still preserving the optimal convergence properties of the approximation space.

\begin{assumption}[Nondegeneracy]\label{asm:nondegeneracy}
    The normalized parameters $\theta_i=(\omega_i,b_i)\in\mathbb S ^d$ satisfy
    \begin{equation}
    \theta_i\neq \theta_j,
    \qquad
    \theta_i\neq -\theta_j,
    \qquad i\neq j,
    \end{equation}
    and each hyperplane
    \begin{equation}
    H_i:=\{x\in\mathbb R^d:s_i(x)=0\}
    \end{equation}
    intersects $\Omega$, i.e., $H_i\cap\Omega\neq \emptyset$.
\end{assumption}

\begin{remark}
The exclusion of antipodal pairs rules out the elementary ReLU–polynomial relation
\begin{equation}
\sigma_r(t)+(-1)^r\sigma_r(-t)=t^r,
\qquad r\in\mathbb N_0.
\end{equation}
Such polynomial components can create additional linear dependencies when several antipodal pairs are present.
\end{remark}

The following result shows that linear independence propagates from a lower-order ReLU NN space to all higher-order ones.
\begin{lemma}[Induction of linear independence]\label{lem:induction}
    Let $r\in\mathbb N_0$ and suppose that
    \begin{equation}
    \{\sigma_r(s_i)\}_{i=1}^n
    \end{equation}
    is linearly independent on $\Omega$. Then for every integer $m\ge r$, the family
    \begin{equation}
    \{\sigma_m(s_i)\}_{i=1}^n
    \end{equation}
    is linearly independent on $\Omega$.
\end{lemma}

\begin{proof}
    Assume
    \begin{equation}
    \sum_{i=1}^n a_i\sigma_m(s_i(x))=0
    \qquad\text{on }\Omega.
    \end{equation}
    Choose $v\in\mathbb  R  ^d$ such that $v\cdot\omega_i\neq0$ for all $i$, this is possible because the finite union of hyperplanes
    $\{v\in\mathbb  R  ^d:v\cdot\omega_i=0\}$ cannot fill $\mathbb  R  ^d$. Applying the directional derivative $D_v^{m-r}$ with respect to $x$ gives
    \begin{equation}
    0=D_v^{m-r}\sum_{i=1}^n a_i\sigma_m(s_i(x))
    =\frac{m!}{r!}\sum_{i=1}^n a_i(v\cdot\omega_i)^{m-r}\sigma_r(s_i(x)).
    \end{equation}
    By the independence of $\{\sigma_r(s_i)\}_{i=1}^n$, each coefficient vanishes:
    \begin{equation}
    a_i(v\cdot\omega_i)^{m-r}=0.
    \end{equation}
    Since $v\cdot\omega_i\neq0$, we conclude $a_i=0$ for all $i$.
\end{proof}

We next make Assumption~\ref{asm:nondegeneracy} more explicit by showing that it provides a sufficient geometric condition for the required linear independence.
\begin{lemma}[A sufficient geometric condition]\label{lem:sufficient}
    Assume that Assumption \ref{asm:nondegeneracy} holds. Then, for every $r\in\mathbb N_0$,
    \begin{equation}
    \operatorname{span}\{\sigma_r(s_i)\}_{i=1}^n\cap \mathcal P _r(\Omega)=\{0\}.
    \end{equation}
    In particular, the family $\{\sigma_r(s_i)\}_{i=1}^n$ is linearly independent on $\Omega$.
\end{lemma}

\begin{proof}
    Suppose
    \begin{equation}
    \sum_{i=1}^n c_i\sigma_r(s_i)+p=0
    \qquad\text{on }\Omega,
    \end{equation}
    with $p\in\mathcal P _r(\Omega)$. Since the $\theta_i$ are normalized and satisfy $\theta_i\neq\pm\theta_j$ for $i\neq j$, the hyperplanes $H_i$ are pairwise distinct.
    
     Choose $v\in\mathbb  R  ^d$ with $v\cdot\omega_i\neq0$ for every $i$. In the distribution sense, we have
    \begin{equation}
    D_v^{r+1}p=0,
    \qquad
    D_v^{r+1}\sigma_r(s_i)=r!(v\cdot\omega_i)^{r+1}\,\delta(s_i).
    \end{equation}
    Hence
    \begin{equation}
    0=D_v^{r+1}\Bigl(\sum_{i=1}^n c_i\sigma_r(s_i)+p\Bigr)
    =r!\sum_{i=1}^n c_i(v\cdot\omega_i)^{r+1}\delta(s_i).
    \end{equation}
    Because the hyperplanes $H_i$ are distinct and each intersects $\Omega$, the distributions $\delta(s_i)$ are linearly independent on $\Omega$. Indeed, for each $i$ one may choose a point
    \begin{equation}
    x_i\in (\Omega\cap H_i)\setminus \bigcup_{j\neq i}H_j
    \end{equation}
    and then a test function supported in a sufficiently small neighborhood of $x_i$, meeting no other hyperplane. Pairing the above identity with such test functions yields
    \begin{equation}
    c_i(v\cdot\omega_i)^{r+1}=0,
    \qquad i=1,\dots,n.
    \end{equation}
    Therefore $c_i=0$ for all $i$, and then also $p=0$.
    The final statement follows by taking $p=0$.
\end{proof}

\begin{remark}
    For the de Rham sequence below, it is enough to assume linear independence of the lowest-order family
    \begin{equation}
    \{\sigma_{k-d}(s_i)\}_{i=1}^n.
    \end{equation}
    Lemma~\ref{lem:sufficient} supplies one sufficient condition, and Lemma~\ref{lem:induction} then lifts the independence to all higher orders.
\end{remark}

\section{3D grad--curl--div complex}\label{sec:3d}
Before proving the arbitrary-dimensional result, it is helpful to illustrate the result in the familiar three-dimensional grad–curl–div setting. Let $d=3$ and assume $k\ge 3$. Define
\begin{equation}
V_n^0=L_n^k,\qquad
V_n^1=(L_n^{k-1})^3,\qquad
V_n^2=(L_n^{k-2})^3,\qquad
V_n^3=L_n^{k-3}.
\end{equation}
Under the standard identification of scalar functions, vector fields, and
differential forms in three dimensions, the exterior derivative corresponds
to the operators $\operatorname{grad}$, $\operatorname{curl}$, and
$\operatorname{div}$. Thus the complex \eqref{eq:relu-complex} takes the
form
\begin{equation}\label{eq:3d-sequence}
0
\longrightarrow V_n^0
\xrightarrow{\operatorname{grad}} V_n^1
\xrightarrow{\operatorname{curl}} V_n^2
\xrightarrow{\operatorname{div}} V_n^3
\longrightarrow 0 .
\end{equation}
The goal is to show that this finite-dimensional grad--curl--div complex is
exact.

The reason for the powers $k,~k-1,~k-2,~k-3$ is already visible on a single neuron.
For constant vectors $a,b\in\mathbb  R^3$, we have
\begin{equation}
\begin{aligned}
    \operatorname{grad} \sigma_k(s_i(x))&=k\sigma_{k-1}(s_i(x))\omega_i,\\
\operatorname{curl}\bigl(\sigma_{k-1}(s_i(x))a\bigr)
&=(k-1)\sigma_{k-2}(s_i(x))(\omega_i\times a),\\
\operatorname{div}\bigl(\sigma_{k-2}(s_i(x))b\bigr)
&=(k-2)\sigma_{k-3}(s_i(x))(\omega_i\cdot b).
\end{aligned}
\end{equation}
So each derivative lowers the ReLU power by one and applies a simple algebraic map to the
constant coefficient attached to that neuron:
\begin{equation}
c\mapsto c\omega_i,\qquad
a\mapsto \omega_i\times a,\qquad
b\mapsto \omega_i\cdot b.
\end{equation}

\begin{proposition}[Exactness of the 3D complex]\label{prop:3d}
    Assume $d=3$, $k\ge3$, and that the family
    \begin{equation}
    \{\sigma_{k-3}(s_i(x))\}_{i=1}^n
    \end{equation}
    is linearly independent on $\Omega$.  Then \eqref{eq:3d-sequence} is exact.
\end{proposition}

\begin{proof}
    By Lemma~\ref{lem:induction}, the families
    $\{\sigma_k(s_i(x))\}$, $\{\sigma_{k-1}(s_i(x))\}$, $\{\sigma_{k-2}(s_i(x))\}$, and
    $\{\sigma_{k-3}(s_i(x))\}$ are all linearly independent.
    
    First, if
    \begin{equation}
    u=\sum_i c_i\sigma_k(s_i(x))\in V_n^0
    \end{equation}
    and $\operatorname{grad} u=0$, then
    \begin{equation}
    0=k\sum_i c_i\sigma_{k-1}(s_i(x))\omega_i.
    \end{equation}
    Independence gives $c_i\omega_i=0$ for every $i$, hence $c_i=0$ since $\omega_i\neq0$. Thus the gradient map is injective.
    
    Next, let
    \begin{equation}
    v=\sum_i \sigma_{k-1}(s_i(x))a_i\in V_n^1
    \end{equation}
    and suppose $\operatorname{curl}v=0$.  Then
    \begin{equation}
    0=(k-1)\sum_i \sigma_{k-2}(s_i(x))(\omega_i\times a_i).
    \end{equation}
    Independence gives $\omega_i\times a_i=0$, so $a_i$ is parallel to $\omega_i$ and
    $a_i=\lambda_i\omega_i$ for some constant $\lambda_i$.  Therefore
    \begin{equation}
    v=\operatorname{grad}\Bigl(\sum_i \frac{\lambda_i}{k}\sigma_k(s_i(x))\Bigr).
    \end{equation}
    
    Now let
    \begin{equation}
    w=\sum_i \sigma_{k-2}(s_i(x))b_i\in V_n^2
    \end{equation}
    and suppose $\operatorname{div}w=0$.  Then
    \begin{equation}
    0=(k-2)\sum_i \sigma_{k-3}(s_i(x))(\omega_i\cdot b_i).
    \end{equation}
    Thus $\omega_i\cdot b_i=0$.  In $\mathbb  R ^3$, this means that there exists a vector $a_i$
    such that $b_i=\omega_i\times a_i$.  Hence
    \begin{equation}
    w=\operatorname{curl}\Bigl(\sum_i \frac{1}{k-1}\sigma_{k-1}(s_i(x))a_i\Bigr).
    \end{equation}
    
    Finally, any element of $V_n^3$ has the form
    \begin{equation}
    f=\sum_i c_i\sigma_{k-3}(s_i(x)).
    \end{equation}
    Choose $b_i\in\mathbb  R  ^3$ with $\omega_i\cdot b_i=c_i$, for instance
    $b_i=c_i\omega_i/|\omega_i|^2$.  Then
    \begin{equation}
    f=\operatorname{div}\Bigl(\sum_i \frac{1}{k-2}\sigma_{k-2}(s_i(x))b_i\Bigr).
    \end{equation}
    This proves exactness of \eqref{eq:3d-sequence}.
\end{proof}

The proof has a clear pattern.  Linear independence separates the different neurons.
For each individual neuron, exactness reduces to elementary linear algebra involving
$\omega_i$, namely the maps $c\mapsto c\omega_i$, $a\mapsto\omega_i\times a$, and
$b\mapsto\omega_i\cdot b$.  The arbitrary-dimensional statement below is the same argument,
with these vector identities replaced by the Koszul complex for the one-form $ds_i$.

\section{Neural de~Rham complex in arbitrary dimension}\label{sec:main}
We now show the exactness in arbitrary dimension. The following elementary result is the standard Koszul homotopy associated with exterior multiplication by a nonzero covector, see, for example, the standard contraction–exterior multiplication identities in \cite[Chapter 1]{bott1982differential}.

\begin{lemma}[Koszul homotopy]\label{lem:koszul}
Let $\lambda\in\Lambda^1$ be a nonzero covector. Choose a vector $v\in\mathbb R^d$ such that $\lambda(v)=1$. Then, for every $p$-form $\alpha\in\Lambda^p$,
\begin{equation}\label{eq:homo}
\iota_v(\lambda\wedge \alpha)+\lambda\wedge \iota_v\alpha=\alpha.
\end{equation}
Here $\iota_v:\Lambda^p\rightarrow\Lambda^{p-1}$ denotes contraction with the vector $v$, namely
\begin{equation}
(\iota_v\alpha)(w_1,\ldots,w_{p-1})=\alpha(v,w_1,\ldots,w_{p-1}).
\end{equation}
Furthermore, the algebraic Koszul complex
\begin{equation}
0\longrightarrow \Lambda^0
\xrightarrow{\lambda\wedge}
\Lambda^1
\xrightarrow{\lambda\wedge}
\cdots
\xrightarrow{\lambda\wedge}
\Lambda^d
\longrightarrow 0
\end{equation}
is exact.
\end{lemma}

\begin{proof}
Since $\lambda\neq 0$, there exists a vector $v\in\mathbb R^d$ such that $\lambda(v)\neq 0$. After rescaling $v$, we may assume that $\lambda(v)=1$. We use the standard contraction–wedge identity
\begin{equation}
\iota_v(\eta\wedge \alpha)=\eta(v)\alpha-\eta\wedge \iota_v\alpha,
\qquad \eta\in\Lambda^1,\quad \alpha\in\Lambda^p.
\end{equation}
Taking $\eta=\lambda$, we obtain
\begin{equation}
\iota_v(\lambda\wedge \alpha)+\lambda\wedge \iota_v\alpha=\alpha.
\end{equation}
We now prove exactness. Suppose first that $1\le p\le d-1$ and that $\alpha\in\Lambda^p$ satisfies
\begin{equation}
\lambda\wedge\alpha=0.
\end{equation}
Applying the identity \eqref{eq:homo} above gives
\begin{equation}
\alpha=\lambda\wedge \iota_v\alpha.
\end{equation}
Thus every element in the kernel of $\lambda\wedge:\Lambda^p\to\Lambda^{p+1}$ lies in the image of $\lambda\wedge:\Lambda^{p-1}\to\Lambda^p$. This proves exactness in the middle degrees.

At degree $0$, if $c\in\Lambda^0=\mathbb R$ and
\begin{equation}
\lambda\wedge c=0,
\end{equation}
then $c\lambda=0$. Since $\lambda\neq0$, we have $c=0$. Hence the first map is injective.

At the top degree, every $\beta\in\Lambda^d$ satisfies
\begin{equation}
\lambda\wedge\beta=0
\end{equation}
since there are no nonzero $(d+1)$-forms on $\mathbb R^d$. Applying the homotopy identity \eqref{eq:homo} to $\beta$ gives
\begin{equation}
\beta=\lambda\wedge\iota_v\beta.
\end{equation}
Thus every top-degree form lies in the image of $\lambda\wedge:\Lambda^{d-1}\to\Lambda^d$. Therefore the complex is exact in all degrees.
\end{proof}

Using the Koszul homotopy lemma, we now prove the exactness of the neural de Rham complex.
\begin{theorem}[Exact ReLU$^k$ neural de Rham complex]\label{thm:exact}
    Assume $k\ge d$ and that the family
    \begin{equation}
    \{\sigma_{k-d}(s_i)\}_{i=1}^n
    \end{equation}
    is linearly independent on $\Omega$, then the complex
    \begin{equation}\label{eq:derham}
        0\longrightarrow\mathcal C ^0 \xrightarrow{d} \mathcal C ^1 \xrightarrow{d}
        \cdots \xrightarrow{d} \mathcal C ^d \longrightarrow 0
    \end{equation}
    is exact.
\end{theorem}

\begin{proof}
By Lemma~\ref{lem:induction}, for every $0\le p\le d$, the family
\begin{equation}
\{\sigma_{k-p}(s_i)\}_{i=1}^n
\end{equation}
is linearly independent on $\Omega$. Hence, every element of $\mathcal C^p$ has a unique representation
\begin{equation}
\nu=\sum_{i=1}^n
\sigma_{k-p}(s_i)\alpha_i,
\qquad
\alpha_i\in\Lambda^p.
\end{equation}
We first prove exactness at $\mathcal C^0$. Assume that $ d\nu=0$, Proposition~\ref{prop:inclusion} gives
\begin{equation}
0= d\nu =k\sum_{i=1}^n
\alpha_i\sigma_{k-1}(s_i)d s_i.
\end{equation}
By the linear independence of $\{\sigma_{k-1}(s_i)\}_{i=1}^n$ and the fact that the coefficient one-forms are constant, we obtain
\begin{equation}
\alpha_i d s_i=0,
\qquad i=1,\ldots,n.
\end{equation}
Since $ d s_i\neq0$, it follows that $\alpha_i=0$ for all $i$. Therefore
\begin{equation}
\nu=0.
\end{equation}
Thus
\begin{equation}
\ker( d:\mathcal C^0\to\mathcal C^1)=0,
\end{equation}
which is exactness at $\mathcal C^0$.

Now let $1\le p\le d-1$, and suppose
\begin{equation}
\nu=\sum_{i=1}^n \sigma_{k-p}(s_i)\alpha_i \in \mathcal C ^p,
\qquad d\nu=0.
\end{equation}
Again by Proposition~\ref{prop:inclusion},
\begin{equation}
0=d\nu=(k-p)\sum_{i=1}^n \sigma_{k-p-1}(s_i)\,ds_i\wedge\alpha_i.
\end{equation}
By independence of $\{\sigma_{k-p-1}(s_i)\}$,
\begin{equation}
ds_i\wedge\alpha_i=0,
\qquad i=1,\dots,n.
\end{equation}
Choose $v_i\in\mathbb  R  ^d$ such that $ds_i(v_i)=1$. Lemma~\ref{lem:koszul} gives
\begin{equation}
\alpha_i=ds_i\wedge \iota_{v_i}\alpha_i.
\end{equation}
Therefore
\begin{equation}
\nu
=\sum_{i=1}^n \sigma_{k-p}(s_i)\,ds_i\wedge \iota_{v_i}\alpha_i
=d\Bigl(\sum_{i=1}^n \frac{1}{k-p+1}\sigma_{k-p+1}(s_i)\,\iota_{v_i}\alpha_i\Bigr).
\end{equation}
Hence $\ker(d:\mathcal C ^p\to\mathcal C ^{p+1})\subset \operatorname{im}(d:\mathcal C ^{p-1}\to\mathcal C ^p)$.
The reverse inclusion is automatic because $d^2=0$.

Finally, let
\begin{equation}
\nu=\sum_{i=1}^n \sigma_{k-d}(s_i)\alpha_i\in\mathcal C ^d,
\qquad \alpha_i\in\Lambda^d.
\end{equation}
Choose $v_i$ with $ds_i(v_i)=1$. Since $\alpha_i$ has top degree, $ds_i\wedge\alpha_i=0$, so Lemma~\ref{lem:koszul} yields
\begin{equation}
\alpha_i=ds_i\wedge\iota_{v_i}\alpha_i.
\end{equation}
Hence
\begin{equation}
\nu
=d\Bigl(\sum_{i=1}^n \frac{1}{k-d+1}\sigma_{k-d+1}(s_i)\,\iota_{v_i}\alpha_i\Bigr),
\end{equation}
which proves surjectivity onto $\mathcal C ^d$.
Thus the whole complex \eqref{eq:derham} is exact.
\end{proof}

\begin{remark}[Relation with the neuron-wise Koszul complex]
    For each fixed neuron $i$, the exactness mechanism is precisely the Koszul complex generated by the nonzero one-form $ds_i$:
    \begin{equation}
    0\to \Lambda^0 \xrightarrow{ds_i\wedge}\Lambda^1\xrightarrow{ds_i\wedge}\cdots
    \xrightarrow{ds_i\wedge}\Lambda^d\to 0.
    \end{equation}
    The scalar independence hypothesis ensures that the different neurons do not interfere with one another and allows the global complex to split coefficient-wise.
\end{remark}

\begin{remark}[Extension to BGG complexes]
The ReLU de~Rham complex constructed above can also provide a natural starting point for Bernstein--Gelfand--Gelfand (BGG) type constructions \cite{arnold2021complexes}. In the BGG framework, one begins with a de~Rham complex, combines the exterior derivative with an algebraic differential, and then passes to suitable kernels or quotient spaces to obtain derived complexes, such as elasticity complexes. The construction and exactness proof developed here suggest a possible route toward ReLU analogues of BGG-derived complexes based on this mechanism.
\end{remark}

\section{Numerical experiments}\label{sec:numerical}
In this section, we present numerical experiments for Galerkin discretizations built from the proposed neural de~Rham complex.  The purpose of these tests is twofold.  First, we examine whether the compatible neural spaces provide stable discrete pairs in mixed formulations.  Second, we test whether the approximation behavior predicted by the underlying linearized ReLU$^k$ spaces is reflected in the numerical errors for representative PDE problems.  The experiments are intended as numerical evidence for the compatibility and approximation properties of the proposed spaces.

\subsection{Mixed formulation of Poisson problem}
\label{subsec:mixed-poisson}

Let \(\Omega=(0,1)^2\).  We consider the  Poisson problem in a mixed formulation
\begin{equation}
\label{eq:mixed-poisson-strong}
    \begin{aligned}
       \boldsymbol{\sigma}-\nabla u &=0
       &&\text{in }\Omega,\\
       -\operatorname{div}\boldsymbol{\sigma} &= f
       &&\text{in }\Omega,\\
       u &=0
       &&\text{on }\partial\Omega .
    \end{aligned}
\end{equation}
The homogeneous Dirichlet condition on \(u\) is natural in the mixed formulation below. We use the solution
\[
    u(x,y)=\sin(\pi x)\sin(\pi y).
\]
The mixed weak formulation is: find
\[
    (\boldsymbol{\sigma},u)\in H(\operatorname{div};\Omega)\times L^2(\Omega)
\]
such that
\begin{equation}
\label{eq:mixed-poisson-weak}
\begin{aligned}
    (\boldsymbol{\sigma},\boldsymbol{\tau})_{\Omega}
    +(u,\operatorname{div}\boldsymbol{\tau})_{\Omega}&=0,
    \qquad
    &&\forall \boldsymbol{\tau}\in H(\operatorname{div};\Omega),\\
    (\operatorname{div}\boldsymbol{\sigma},v)_{\Omega}
    &=-(f,v)_{\Omega},
    \qquad
    &&\forall v\in L^2(\Omega).
    \end{aligned}
\end{equation}
For an integer \(k\ge 1\), we use the neural spaces
\[
    \Sigma_{n,k}=(L_n^k)^2,
    \qquad
    W_{n,k}=L_n^{k-1}.
\]
This pair is the \(H(\operatorname{div})\)-\(L^2\) part of the two-dimensional neural de Rham complex. The discrete problem is: find
\[
    (\boldsymbol{\sigma}_{n},u_n)\in \Sigma_{n,k}\times W_{n,k}
\]
such that
\begin{equation}
\label{eq:mixed-poisson-discrete}
\begin{aligned}
    (\boldsymbol{\sigma}_{n},\boldsymbol{\tau}_{n})_{\Omega}
    +(u_n,\operatorname{div}\boldsymbol{\tau}_{n})_{\Omega}&=0,
    \qquad
    &&\forall \boldsymbol{\tau}_{n}\in \Sigma_{n,k},\\
    (\operatorname{div}\boldsymbol{\sigma}_{n},v_n)_{\Omega}
    &=-(f,v_n)_{\Omega},
    \qquad
    &&\forall v_n\in W_{n,k}.
        \end{aligned}
\end{equation}
We report the cases \(k=1\) and \(k=2\).

For smooth solutions and quasi-uniform neuron parameters excluding antipodal pairs, the approximation theory for linearized ReLU$^k$ spaces suggests the rates
\[
\begin{aligned}
    \|\boldsymbol{\sigma}-\boldsymbol{\sigma}_{n}\|_{H(\operatorname{div};\Omega)}
    &=
    O\left(n^{-\frac12-\frac{2k-1}{2d}}\right),\\
    \|\boldsymbol{\sigma}-\boldsymbol{\sigma}_{n}\|_{L^2(\Omega)}
    &=
    O\left(n^{-\frac12-\frac{2k+1}{2d}}\right),\\
    \|u-u_n\|_{L^2(\Omega)}
    &=
    O\left(n^{-\frac12-\frac{2k-1}{2d}}\right),
    \end{aligned}
\]
with \(d=2\) in the present experiment. We also compute the discrete inf-sup constant
\begin{equation}
\label{eq:discrete-infsup}
    \beta_{n,k}
    =
    \inf_{0\neq v_n\in W_{n,k}}
    \sup_{0\neq \boldsymbol{\tau}_n\in \Sigma_{n,k}}
    \frac{
    (\operatorname{div}\boldsymbol{\tau}_n,v_n)_{\Omega}
    }{
    \|\boldsymbol{\tau}_n\|_{H(\operatorname{div};\Omega)}
    \|v_n\|_{L^2(\Omega)}
    } .
\end{equation}
This quantity measures the stability of the discrete divergence pairing.  In particular, exactness of the neural complex implies the algebraic surjectivity of the discrete divergence map, while the behavior of \(\beta_{n,k}\) gives numerical information about the strength of this stability. A uniform lower bound for $\beta_{n,k}$ would require an additional stability estimate, such as a uniformly bounded right inverse of the discrete divergence operator or a bounded commuting projection. Although this is not proved here, the numerical evidence in Table~\ref{tab:mixpoissoninfsup} shows that these values remain bounded away from zero as $n$ increases, which is consistent with the surjectivity of the discrete divergence operator.

\begin{table}[!htbp]
\centering
\setlength{\tabcolsep}{1cm}
\caption{Discrete $H(\operatorname{div})$--$L^2$ inf-sup constants $\beta_{n,k}$.}
\label{tab:mixpoissoninfsup}
\begin{tabular}{c|cc}
\hline
$n$ & $k=1$ & $k=2$ \\
\hline
16 & 9.753e-01 & 4.346e-01 \\
32 & 9.665e-01 & 9.639e-01 \\
64 & 9.759e-01 & 9.756e-01 \\
128 & 9.757e-01 & 9.756e-01 \\
256 & 9.756e-01 & 9.756e-01 \\
\hline
\end{tabular}
\end{table}

Figure~\ref{fig:mixed-poisson} reports the errors for the pairs $(L_n^1)^2\times L_n^0$ and $(L_n^2)^2\times L_n^1 $. The observed convergence is consistent with the expected approximation behavior of the underlying linearized ReLU spaces.

\begin{figure}[!htbp]
    \centering
    \includegraphics[width=0.45\textwidth]{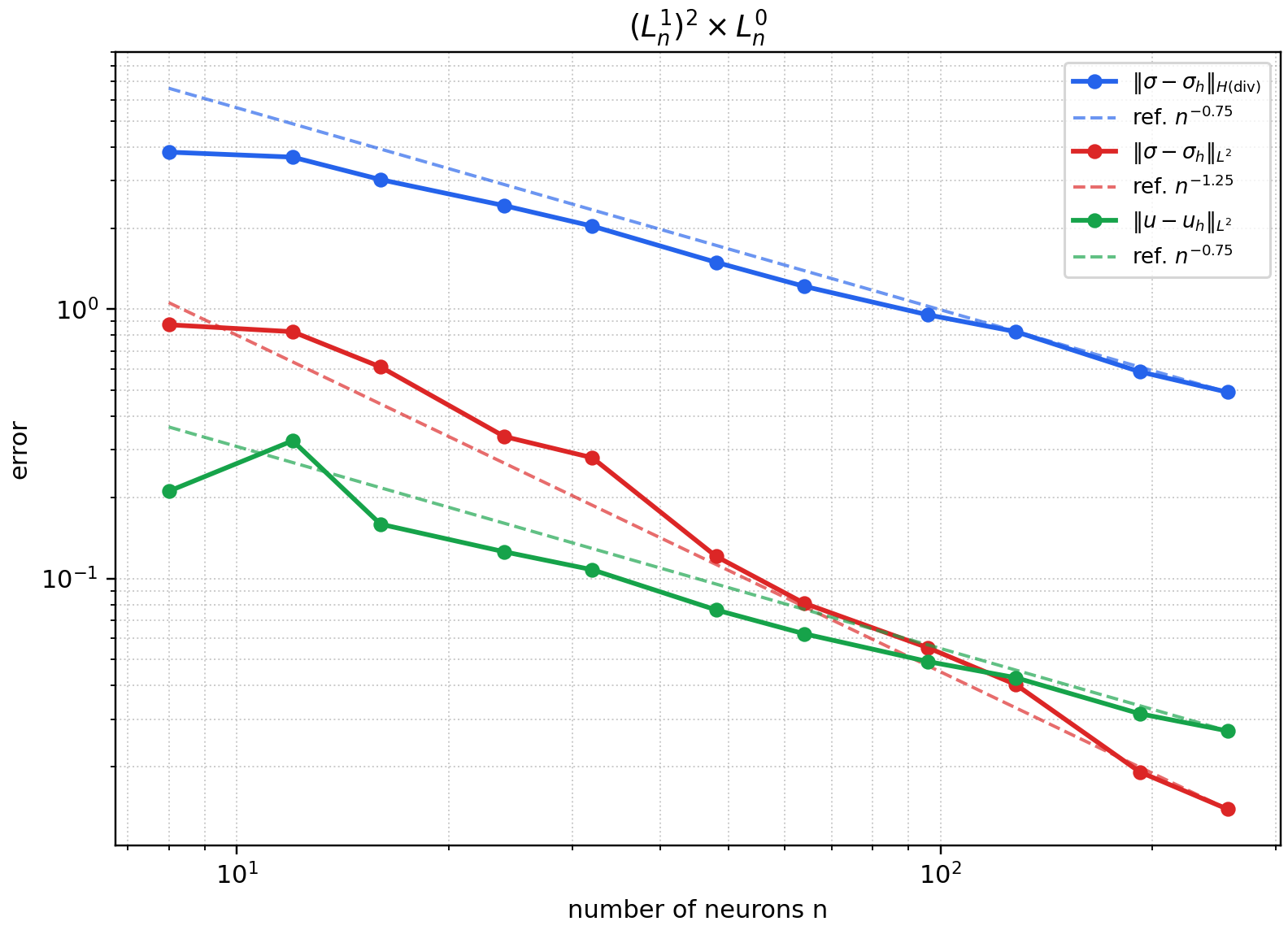}
    \includegraphics[width=0.45\textwidth]{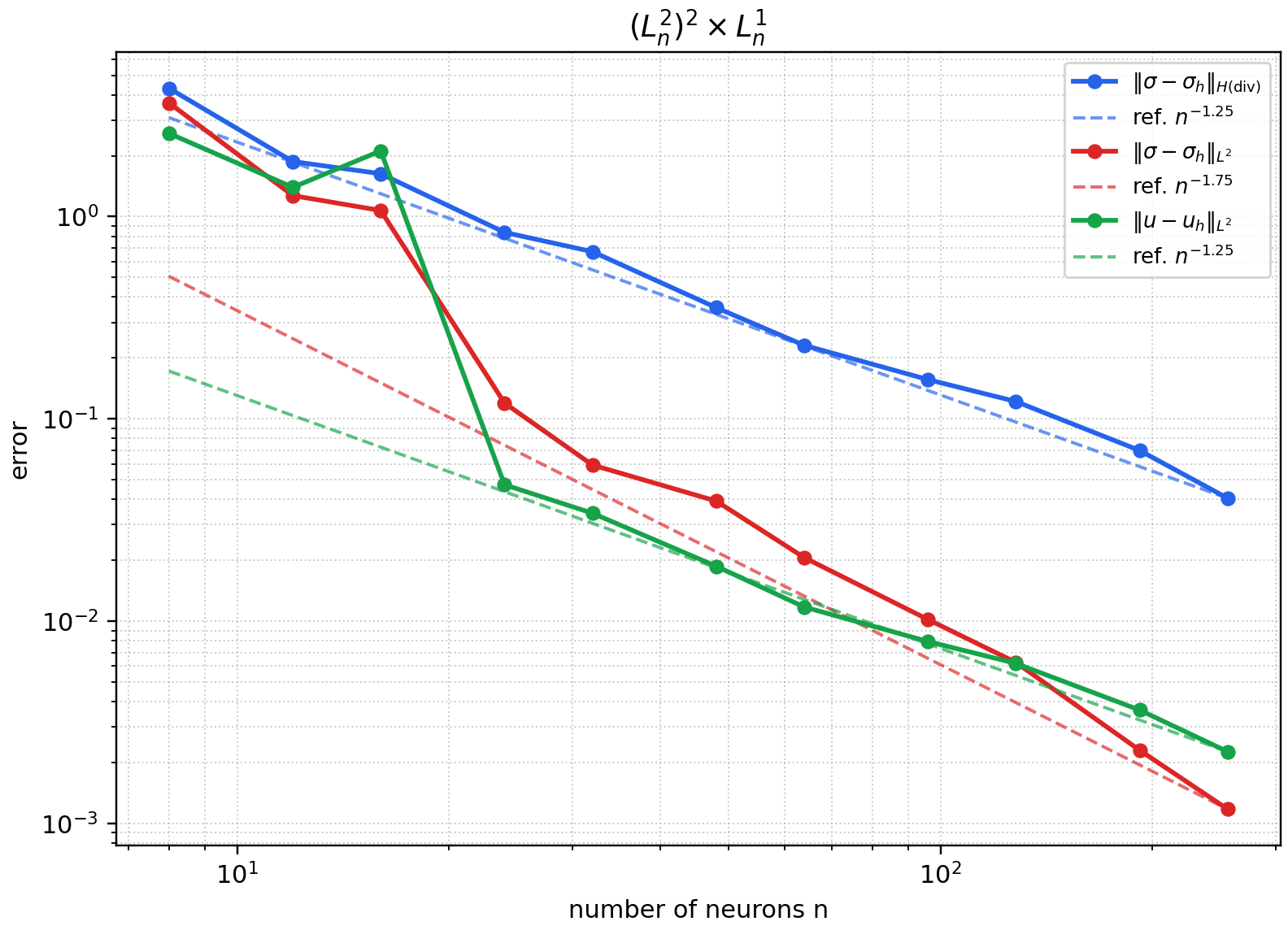}
    \caption{Errors of mixed formulation of Poisson problem for \((L_n^1)^2\times L_n^0\) (left) and
    \((L_n^2)^2\times L_n^1\) (right).}
    \label{fig:mixed-poisson}
\end{figure}

\subsection{Hodge Laplacian problem}\label{subsec:hodge-laplacian}
We next test the three-dimensional Hodge Laplacian problem for one-forms on the unit cube $\Omega=(0,1)^3$
\begin{equation}
\label{eq:hodge-strong}
    -\nabla\nabla\cdot\boldsymbol u
    +
    \nabla\times\nabla\times \boldsymbol u
    =
    \boldsymbol f
    \qquad\text{in }\Omega ,
\end{equation}
with boundary conditions
\begin{equation}
\label{eq:hodge-natural-bc}
    \boldsymbol u\cdot n=0,
    \qquad
    n\times(\nabla\times\boldsymbol u)=0
    \qquad\text{on }\partial\Omega .
\end{equation}
Introducing
\[
    \sigma=\delta \boldsymbol u=-\operatorname{div}\boldsymbol u,
\]
the mixed formulation is: find
\[
    (\sigma,\boldsymbol u)\in H^1(\Omega)\times H(\operatorname{curl};\Omega)
\]
such that
\begin{equation}
\label{eq:hodge-mixed}
\begin{aligned}
    (\sigma,\tau)_\Omega
    -
    (\boldsymbol u,\nabla\tau)_\Omega
    &=0,
    \qquad
    &&\forall \tau\in H^1(\Omega),\\
    (\nabla\sigma,\boldsymbol v)_\Omega
    +
    (\nabla\times\boldsymbol u,\nabla\times\boldsymbol v)_\Omega
    &=
    (\boldsymbol f,\boldsymbol v)_\Omega,
    \qquad
    &&\forall \boldsymbol v\in H(\operatorname{curl};\Omega).
\end{aligned}
\end{equation}
For the natural boundary conditions used in the Hodge Laplacian problem, the
degree-zero cohomology contains the constant functions. To represent this
component correctly at the discrete level, we augment the degree-zero neural
space by constants. Thus, in the three-dimensional neural de~Rham sequence,
we use the discrete spaces
\[
    \sigma_n\in \mathbb R\oplus L_n^k,
    \qquad
    \boldsymbol u_n\in (L_n^{k-1})^3 ,
\]
for $k=2,3$.

We use the manufactured solution $\bm u=(u_1,u_2,u_3)^T$ with
\[
\begin{aligned}
    u_1(x,y,z)
    &=
    \sin(\pi x)\cos(\pi y)\cos(\pi z),\\
    u_2(x,y,z)
    &=
    \frac12\cos(\pi x)\sin(\pi y)\cos(\pi z),\\
    u_3(x,y,z)
    &=
    \frac14\cos(\pi x)\cos(\pi y)\sin(\pi z).
\end{aligned}
\]
We report the $L^2$ and graph-norm errors for both $\sigma$ and $\bm u$. Figure~\ref{fig:hodge-3d-one-form} shows the convergence history.  The errors decrease consistently as the number of neurons increases, consistent with the optimal convergence rate.  This experiment illustrates that the proposed neural complex can be used in a mixed Hodge Laplacian formulation. 

\begin{figure}[!htbp]
    \centering
    \includegraphics[width=0.45\textwidth]{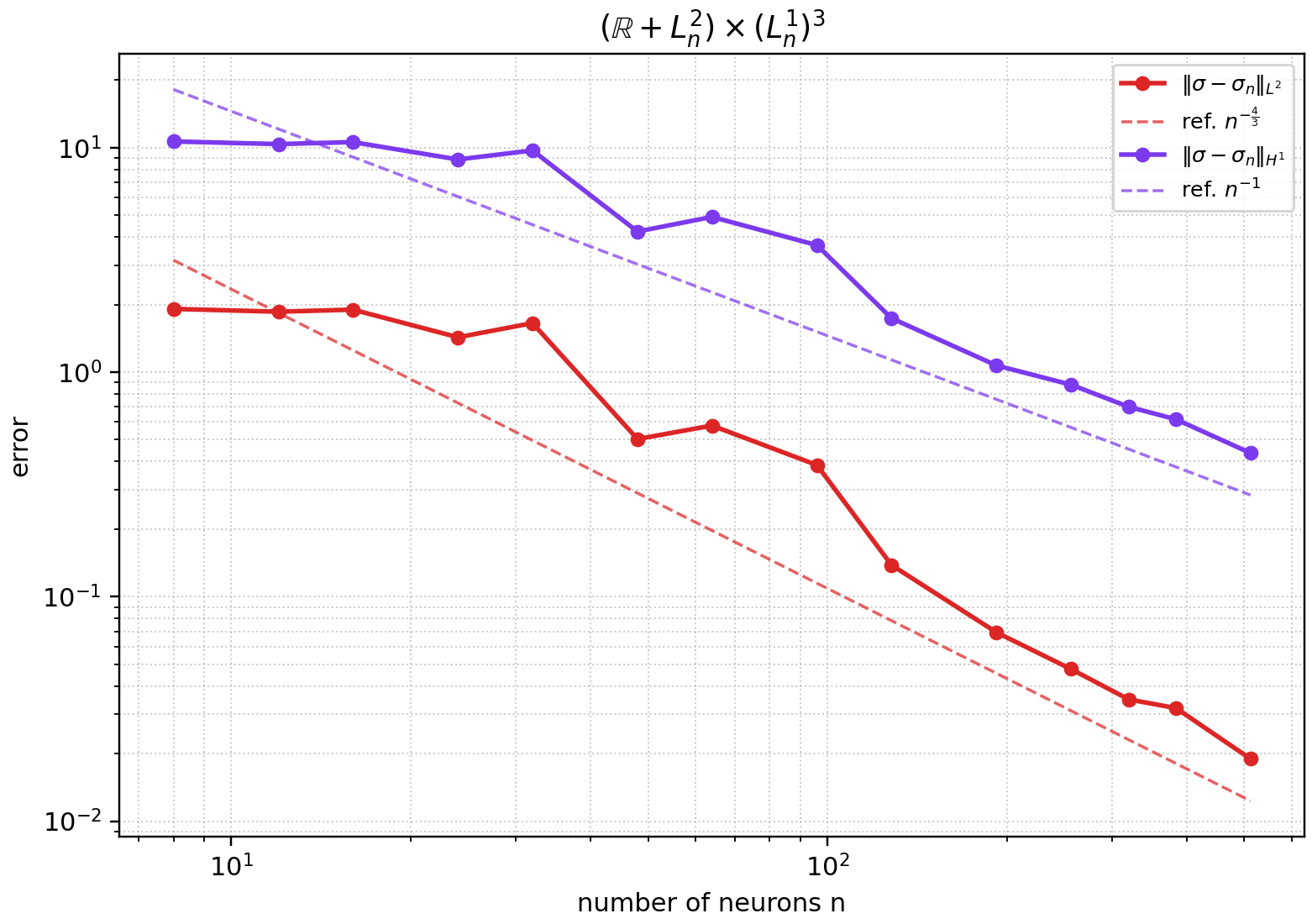}
    \includegraphics[width=0.45\textwidth]{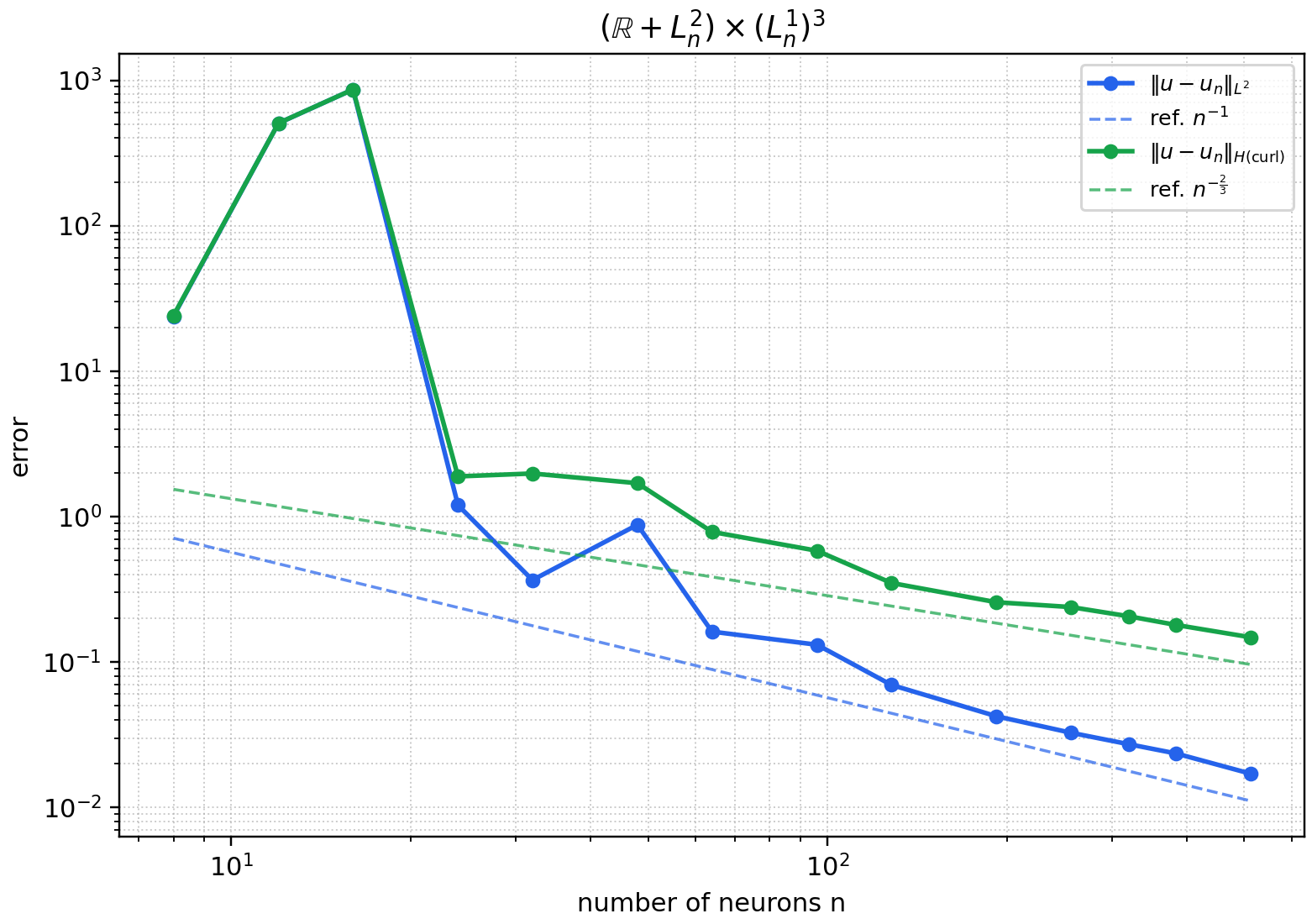}\\
     \includegraphics[width=0.45\textwidth]{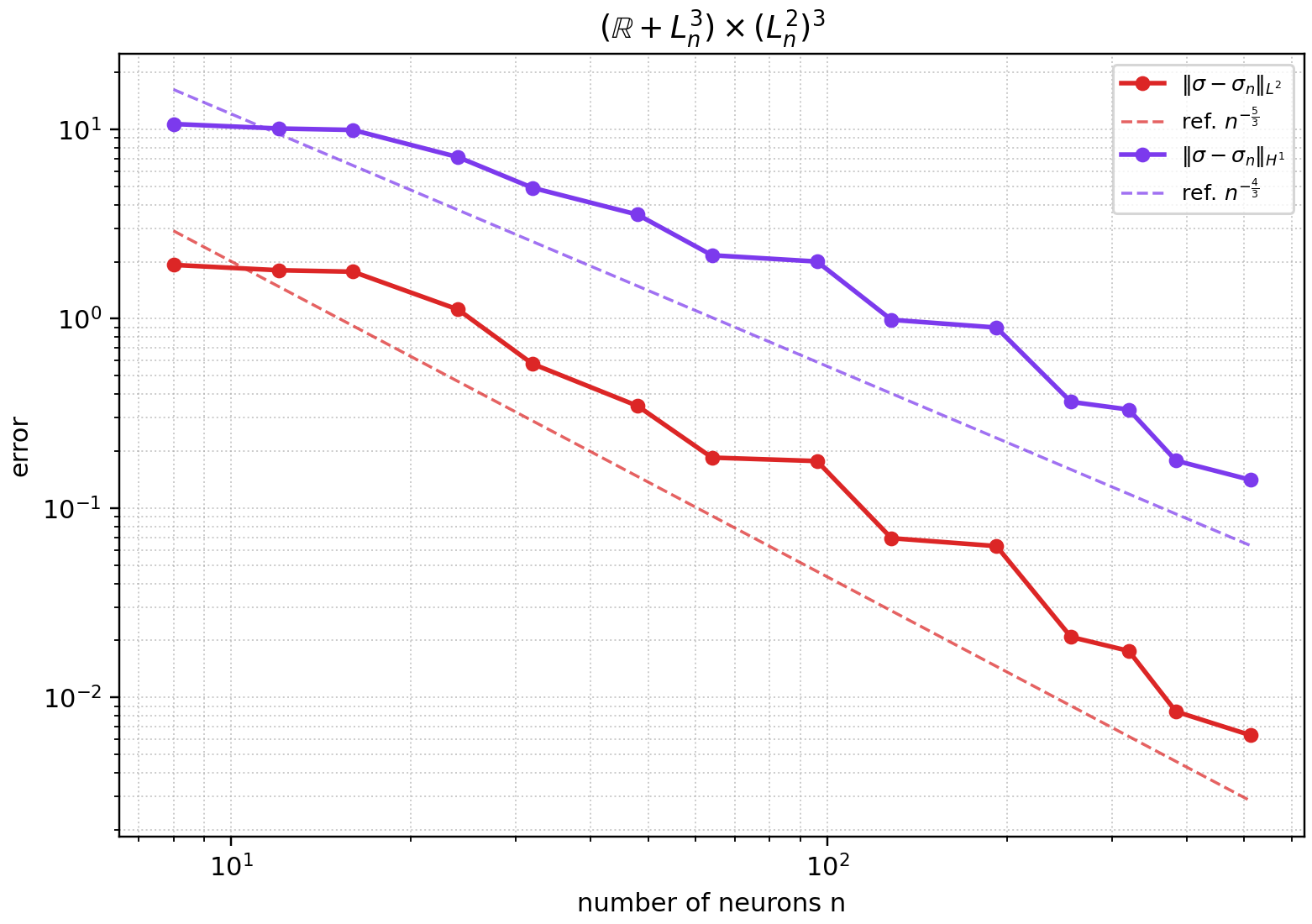}
    \includegraphics[width=0.45\textwidth]{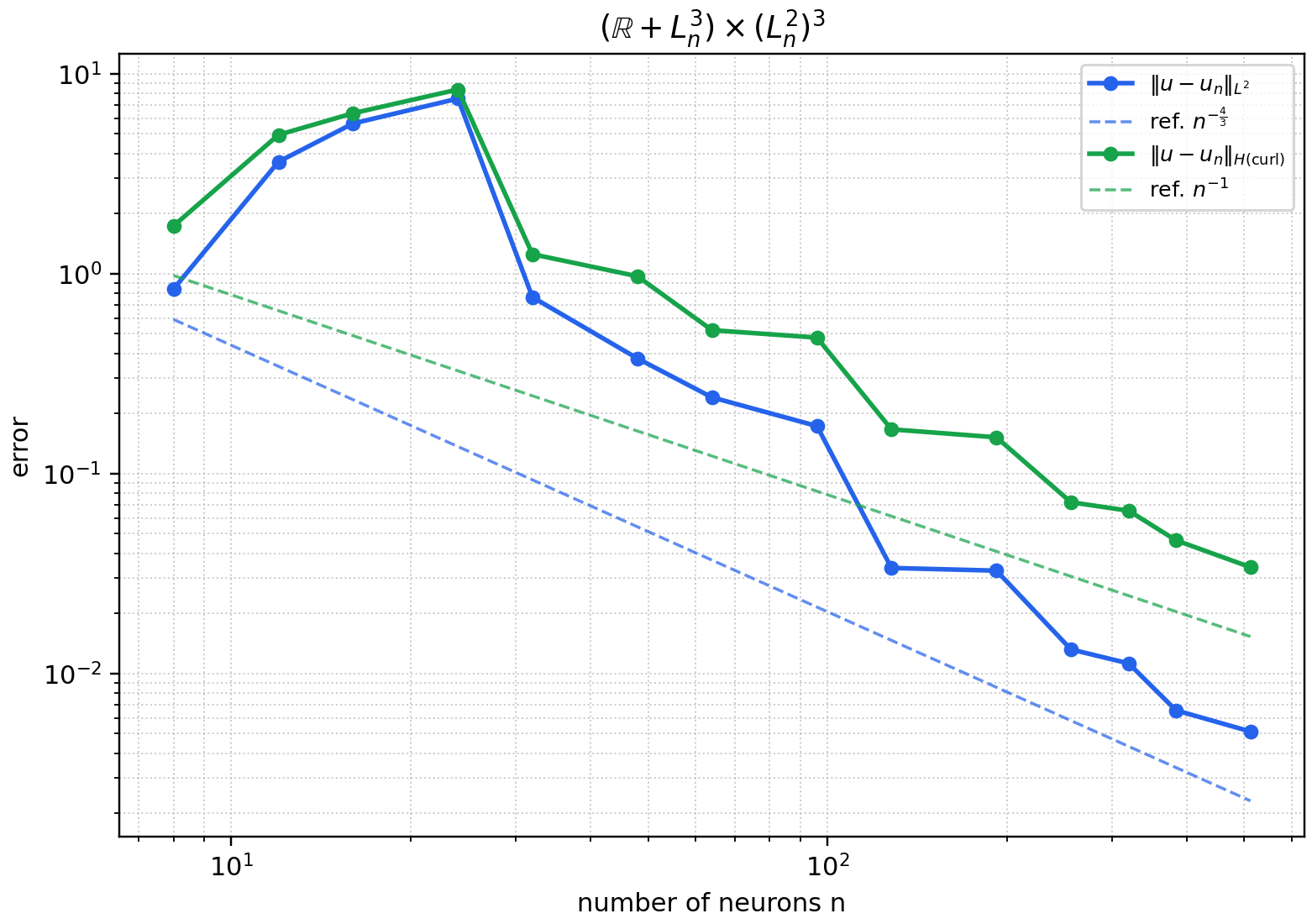}
    \caption{Three-dimensional one-form Hodge Laplacian errors for
    \((\mathbb R\oplus L_n^2)\times (L_n^1)^3\) (top) and \((\mathbb R\oplus L_n^3)\times (L_n^2)^3\) (bottom).  The left column shows the
    \(\sigma\)-errors in \(L^2\) and \(H^1\), and the right column shows the
    \(\boldsymbol u\)-errors in \(L^2\) and \(H(\operatorname{curl})\).}
    \label{fig:hodge-3d-one-form}
\end{figure}

\subsection{Hodge Laplacian eigenvalue problem on an L-shaped domain}
\label{subsec:lshape-hodge-laplacian-eigenvalue}

We consider the Hodge Laplacian eigenvalue problem on the L-shaped domain
\[
\Omega=(-1,1)^2\setminus\bigl((0,1)\times(-1,0)\bigr).
\]
For a two-dimensional one-form, identified with a vector field
\(\boldsymbol u=(u_1,u_2)\), the eigenvalue problem can be written as
\[
-\nabla\operatorname{div}\boldsymbol u
+
\operatorname{curl}\operatorname{curl}\boldsymbol u
=
\lambda \boldsymbol u ,
\]
with boundary conditions $\boldsymbol u\cdot\boldsymbol n=0$ and
$\operatorname{curl}\boldsymbol u=0$ on $\partial\Omega$. Here the two curl operators are different: the inner one is the scalar
curl of a vector field,
$\operatorname{curl}\boldsymbol u=\partial_1u_2-\partial_2u_1$, and the
outer one is the vector curl of a scalar,
$\operatorname{curl}\phi=(\partial_2\phi,-\partial_1\phi)^T$, and accordingly
$H(\operatorname{curl};\Omega)=\{\boldsymbol v\in L^2(\Omega)^2:
\operatorname{curl}\boldsymbol v\in L^2(\Omega)\}$. The re-entrant corner of the L-shaped domain gives rise to singular eigenfunctions with limited Sobolev regularity.  In particular, the eigenfunction associated with the first eigenvalue belongs only to $H^{2/3-\varepsilon}(\Omega)$ for every $\varepsilon>0$.
For the linearized ReLU$^k$ spaces with quasi-uniform neurons, the
approximation of such a field in the $H(\operatorname{curl})$ norm is
therefore limited by the regularity, and one
expects at most the rate $n^{-1/3+\varepsilon}$ in terms of the number of
neurons $n$.  

We compare two variational discretizations: a primal formulation and a
mixed formulation, while the former is prone to spurious modes, whereas the latter, built on the proposed ReLU$^k$ neural complex, reproduces the spectrum correctly.

\textit{Primal formulation.} The primal formulation is: find
\((\lambda,\boldsymbol u)\in\mathbb R\times V\), where
\[
  V=
  \left\{
  \boldsymbol v\in H(\operatorname{curl};\Omega)\cap
  H(\operatorname{div};\Omega):
  \boldsymbol v\cdot \boldsymbol n=0
  \text{ on }\partial\Omega
  \right\},
\]
such that
\[
  (\operatorname{curl}\boldsymbol u,\operatorname{curl}\boldsymbol v)
  +
  (\operatorname{div}\boldsymbol u,\operatorname{div}\boldsymbol v)
  =
  \lambda(\boldsymbol u,\boldsymbol v),
  \qquad \forall \boldsymbol v\in V .
\]
We use the neural trial space
\[
V_n^k=(L_n^k)^2 .
\]
The normal trace condition is imposed strongly. More precisely, on each
boundary edge, the restriction of every ridge function is a piecewise
polynomial. We subdivide each edge at the ridge breakpoints and form the
exact trace matrix \(T\) by requiring all polynomial coefficients of
\(\boldsymbol v_n\cdot\boldsymbol n\) to vanish on each subinterval. The
actual primal trial space is therefore
\[
  V_{n,0}^k
  =
  \{\boldsymbol v_n\in (L_n^k)^2:T\boldsymbol v_n=0\}.
\]

\textit{Mixed formulation.}
The mixed formulation is: find
\[
(\lambda,\sigma,\boldsymbol u)\in
\mathbb R\times H^1(\Omega)\times H(\operatorname{curl};\Omega)
\]
such that
\[
\begin{aligned}
  (\sigma,\tau)-(\boldsymbol u,\nabla\tau)&=0,
  \qquad &&\forall \tau\in H^1(\Omega),\\
  (\nabla\sigma,\boldsymbol v)
  +
  (\operatorname{curl}\boldsymbol u,\operatorname{curl}\boldsymbol v)
  &=
  \lambda(\boldsymbol u,\boldsymbol v),
  \qquad &&\forall \boldsymbol v\in H(\operatorname{curl};\Omega).
  \end{aligned}
\]
For the same reason, we augment the
degree-zero neural space by constants and use the discrete sequence
\[
  \mathbb R\oplus L_n^{k+1}
  \xrightarrow{\nabla}
  (L_n^k)^2 .
\]

By the eigenvalue approximation theory for mixed
formulations~\cite{BABUSKA1991641,boffi2023convergence}, the
eigenvalue error converges at twice the eigenfunction rate. For the
singular eigenvalues this yields
\[
|\lambda_j-\lambda_{j,n}|
\ \lesssim\
n^{-2/3+\varepsilon} .
\]
For eigenvalues with smooth eigenfunctions the error is instead limited
by the approximation power of the neural spaces: the mixed eigenvalue
error doubles the optimal $L^2$ approximation rate
$n^{-1/2-(2k+1)/(2d)}$ of $(L_n^k)^2$, giving $n^{-(k+3/2)}$ in $d=2$,
whereas the primal eigenvalue error doubles the energy-norm rate
$n^{-1/2-(2k-1)/(2d)}$, giving $n^{-(k+1/2)}$.  We therefore use
$n^{-2/3}$ as the reference slope for the eigenvalues associated with
singular eigenfunctions, and $n^{-(k+3/2)}$ (mixed) and $n^{-(k+1/2)}$
(primal) for the smooth ones.

Figure~\ref{fig:lshape-hodge-eigenvalue} compares the computed spectra of
the mixed and primal formulations for \(k=1,2\) as the number of neurons
increases.  The red dashed lines indicate the reference continuous
eigenvalues, and the blue markers the computed discrete eigenvalues.  The
mixed formulation based on the neural complex tracks the reference
spectrum without spurious modes.  By contrast, the primal formulation
produces eigenvalues lying in spectral gaps and, in particular, has no
discrete eigenvalue near the singular eigenvalue
\(\lambda_1=1.4756\).

Figure~\ref{fig:lshape-hodge-lambda1-rate} shows the convergence histories
for a singular and a smooth eigenvalue for \(k=1,2\).  For the mixed
formulation the observed rates are close to the corresponding reference
slopes: \(n^{-2/3}\) for the singular mode caused by the re-entrant
corner, and \(n^{-k-3/2}\) for the smooth mode.  For the primal
formulation, since no discrete eigenvalue approaches \(\lambda_1\), the
quantity reported for the singular mode is the distance from \(\lambda_1\) to the discrete spectrum; it stagnates at \(O(1)\),
quantifying the missing mode. For the smooth eigenvalue the primal method does
converge, but one order slower than the mixed method, at the rate
$n^{-(k+1/2)}$ governed by the energy norm, in agreement with the
reference slopes in Figure~\ref{fig:lshape-hodge-lambda1-rate}. Together with the spectral
plots in Figure~\ref{fig:lshape-hodge-eigenvalue}, these results indicate that the mixed discretization based on the proposed neural complex tracks the correct spectrum, resolves the singular eigenvalues at the rate permitted by their regularity, and avoids spurious eigenvalues.

\begin{figure}[!htbp]
    \centering
    \includegraphics[width=0.7\linewidth]{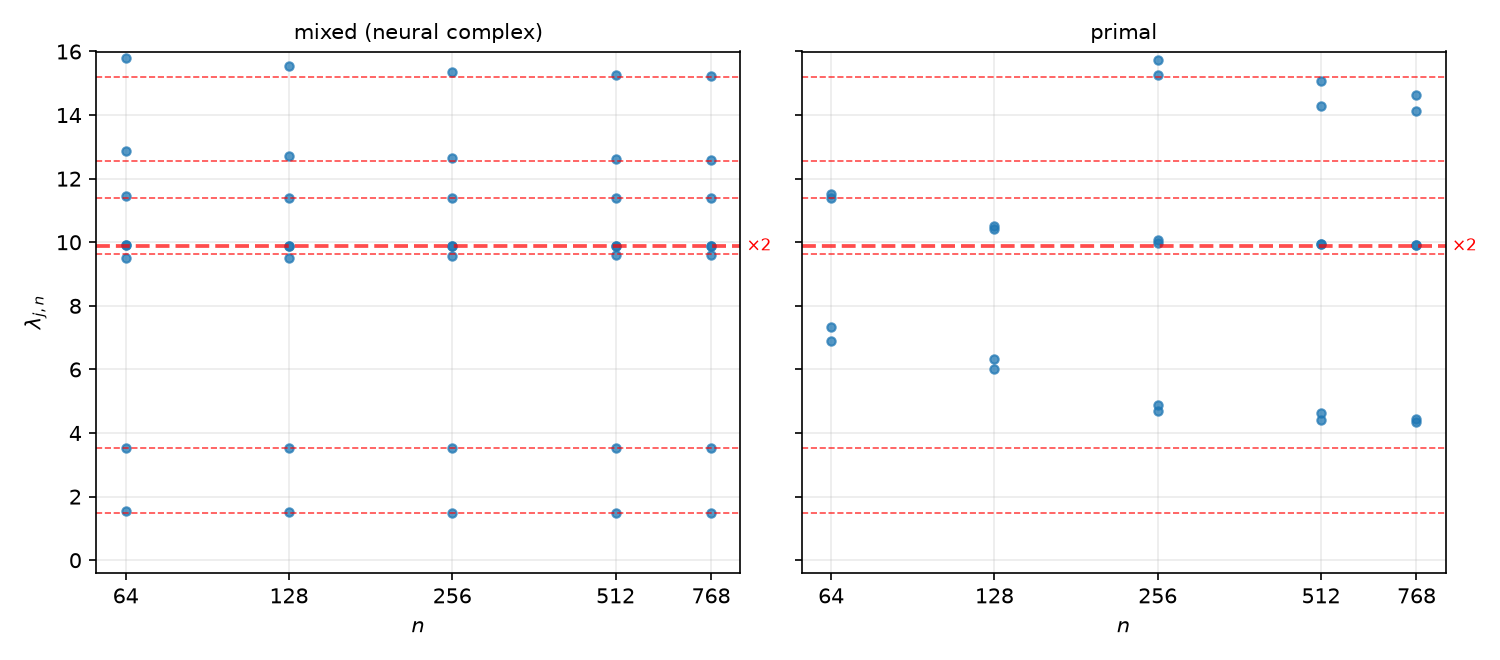}
     \includegraphics[width=0.7\linewidth]{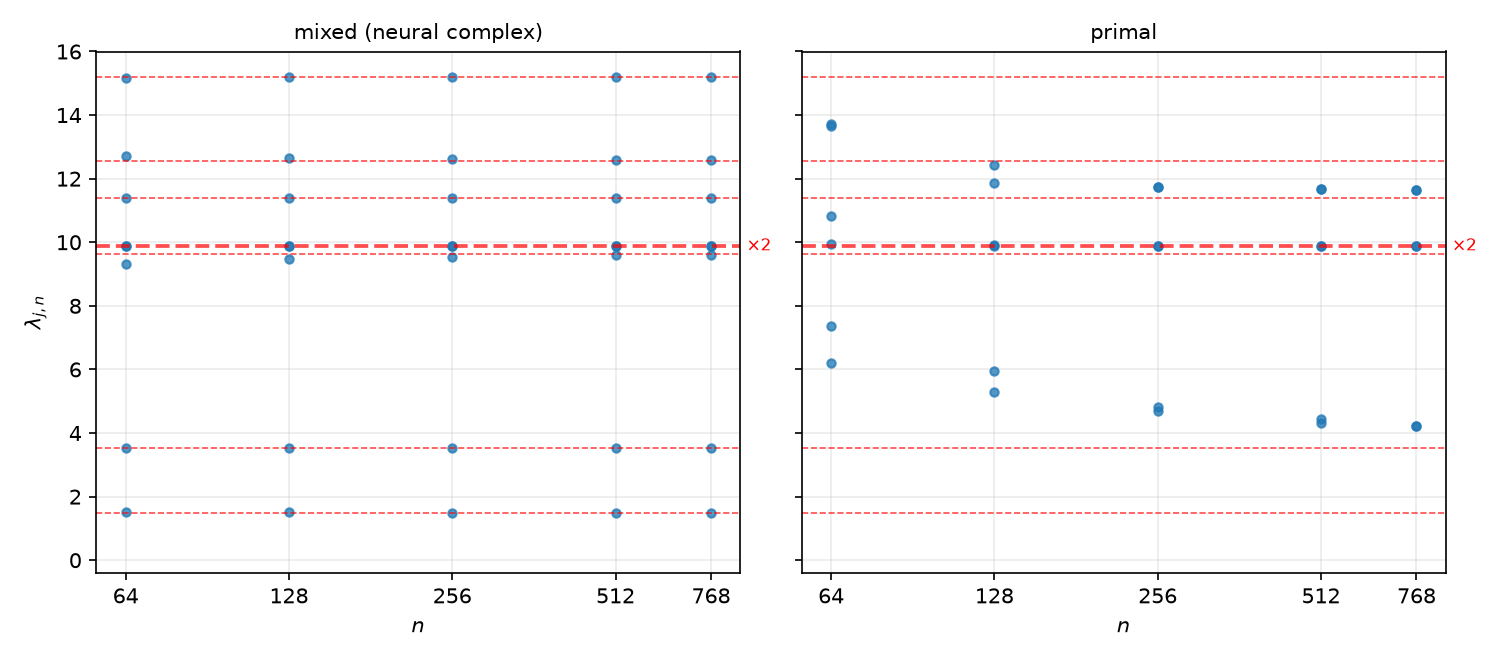}
    \caption{Computed spectra for the L-shaped-domain Hodge–Laplacian eigenvalue problem. The top and bottom panels show the cases $k=1$ and $k=2$, respectively. In each panel, the mixed neural-complex formulation is shown on the left and the primal formulation on the right. Blue markers denote computed eigenvalues, and red dashed lines denote reference eigenvalues.}
    \label{fig:lshape-hodge-eigenvalue}
\end{figure}

\begin{figure}[!htbp]
\centering
\includegraphics[width=.7\linewidth]{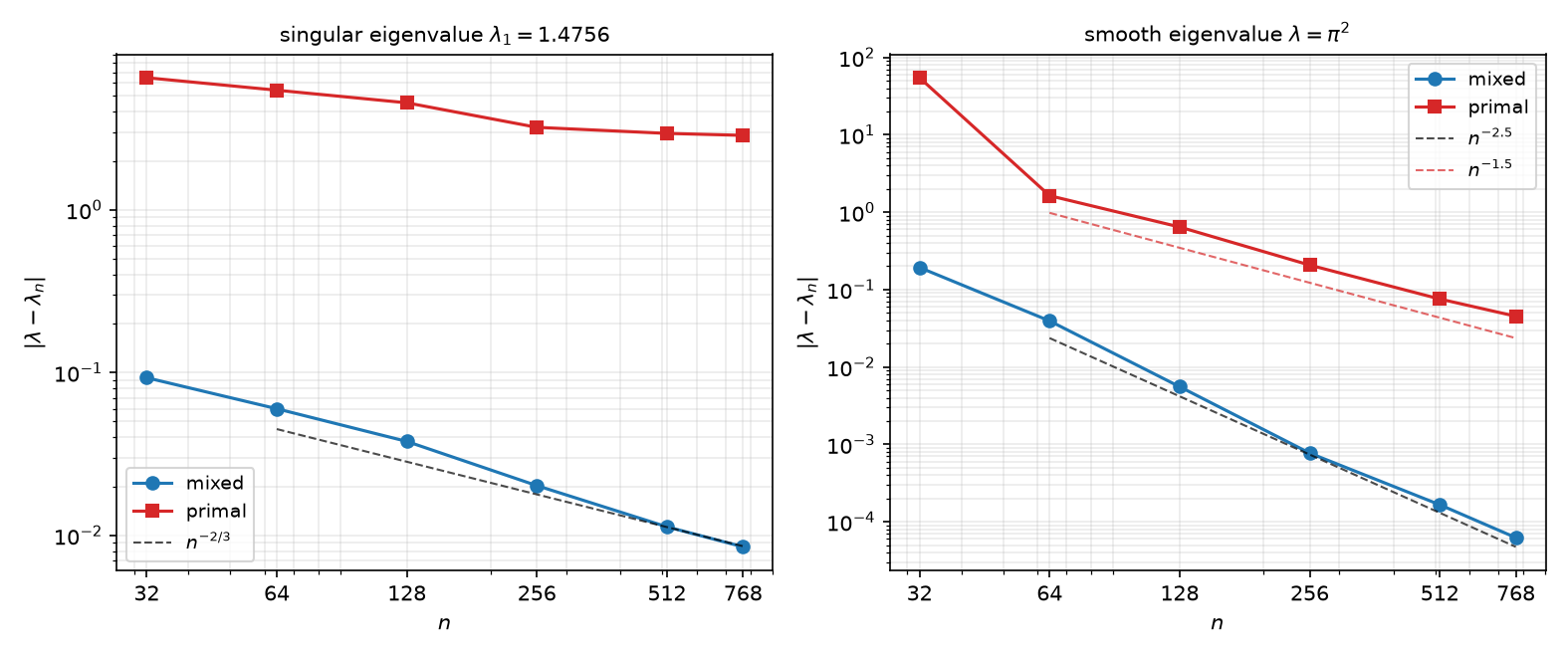}
\includegraphics[width=.7\linewidth]{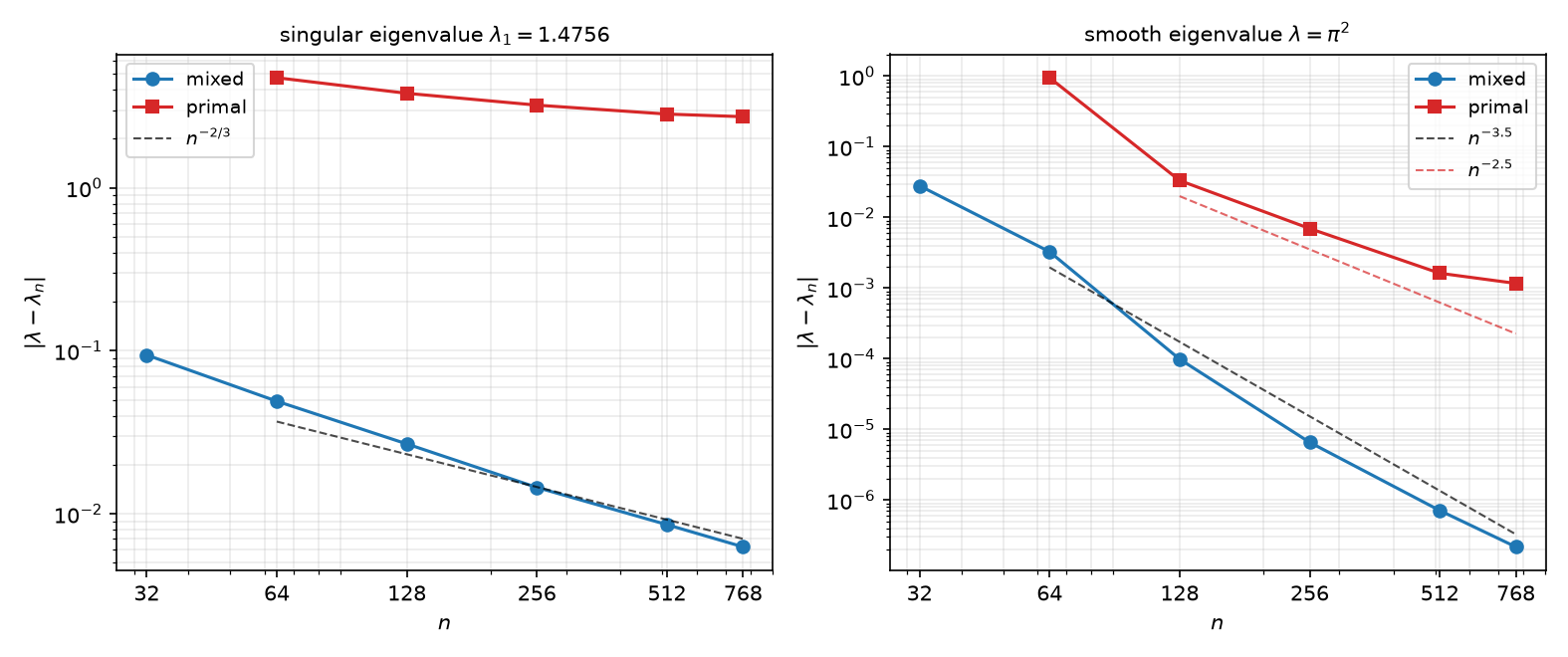}
\caption{Log–log convergence plots for the L-shaped-domain Hodge–Laplacian eigenvalue problem. The left and right columns show singular and smooth eigenvalues, respectively, while the top and bottom rows correspond to $k=1$ and $k=2$.}
\label{fig:lshape-hodge-lambda1-rate}
\end{figure}

\section{Concluding remarks and future work}
\label{sec:discussion}

In this paper we constructed an exact ReLU$^k$ neural de Rham subcomplex generated by fixed neurons.  The construction is straightforward and relies on two simple ingredients: linear independence of the neurons and the neuron-wise Koszul exactness.  Under the lowest-order linear independence assumption, the exterior derivative separates the different neurons and reduces the exactness proof to elementary algebra. In this sense, the resulting complex provides a mesh-free analogue of a finite-dimensional de~Rham subcomplex inside a linearized shallow ReLU$^k$ neural network space. The numerical experiments also provide numerical evidence for the usefulness of the compatible construction.

The significance of the construction is twofold.  First, it gives a compatible differential structure for fixed-neuron ReLU$^k$ spaces. Second, when combined with existing approximation results for linearized ReLU$^k$ networks, it suggests a possible route toward structure-preserving neural discretizations of partial differential equations.  Exactness gives the correct discrete kernel--range identities and rules out algebraic spurious cohomology within the constructed complex.  These properties are precisely the algebraic ingredients that underlie stable mixed finite element methods and FEEC.  

The algebraic exactness established in this paper provides a foundation for several further developments. A natural first direction is the construction of boundary-compatible neural complexes. The spaces considered here are generated by unrestricted ridge functions and constant-coefficient differential forms, which makes them flexible approximation spaces in the interior of the domain. For PDE applications, it would be useful to construct neural subspaces whose tangential or normal traces satisfy the prescribed boundary conditions by design. Weak enforcement through penalty or least-squares terms provides a practical starting point. 

A second direction is to extend the construction to domains with nontrivial topology. The complex studied here is algebraically exact and therefore provides an appropriate model for contractible domains. On nontrivial domains, a structure-preserving neural complex should contain suitable harmonic representatives or additional degrees of freedom. The existence of a bounded cochain interpolation operator for the present spaces also remains open. One possible direction is to combine the present neuron-wise construction with the finite element system framework~\cite{christiansen2018generalized}, where local compatibility and global cohomology can be treated systematically. 

A third direction is to investigate other activation families for which differentiation maps one level of the associated function hierarchy into another, and whose linearized spaces possess suitable approximation and linear-independence properties. Fourier modes on a torus provide a particularly natural example. For every nonzero Fourier frequency, the corresponding de Rham sequence reduces to an exact Koszul complex, whereas the zero-frequency component represents the nontrivial cohomology of the torus. This decomposition suggests a possible route toward neural or spectral complexes that simultaneously encode local exactness and global topology.

These questions indicate that exactness is only the first step toward a full FEEC-type theory for neural network discretizations.  The present construction provides a simple algebraic model showing that shallow ReLU$^k$ neural spaces can be organized into a compatible de Rham sequence.  Developing boundary-compatible spaces, topology-aware neural complexes, and stability estimates for mixed PDE formulations would turn the present shallow ReLU$^k$ complexes into a systematic framework for structure-preserving neural discretizations of differential-form PDEs.

\section*{Acknowledgment}
The work of KH was supported by a Royal Society University Research Fellowship (URF$\backslash$R1$\backslash$221398) and an ERC Starting Grant (project 101164551, GeoFEM). The work of JW was supported by a Royal Society Newton International Fellowship (NIF$\backslash$R1$\backslash$252881). The work of JX was supported by King Abdullah University of Science and Technology (KAUST) Baseline Research Fund. Views and opinions expressed are however those of the authors only and do not necessarily reflect those of the European Union or the European Research Council. Neither the European Union nor the granting authority can be held responsible for them.

\bibliography{ref}

\end{document}